\documentclass[12pt,a4paper]{article}

\usepackage[T1]{fontenc}
\usepackage{lmodern}
\usepackage[a4paper,margin=25mm]{geometry}
\usepackage{microtype}
\usepackage{amsmath,amssymb,amsthm,mathtools}
\usepackage{array}
\usepackage[colorlinks,linkcolor=blue,citecolor=blue,urlcolor=blue]{hyperref}
\usepackage{bookmark}
\usepackage{aliascnt}
\usepackage[nameinlink,noabbrev]{cleveref}

\allowdisplaybreaks
\numberwithin{equation}{section}

\theoremstyle{plain}
\newtheorem{theorem}{Theorem}[section]

\newaliascnt{lemma}{theorem}
\newtheorem{lemma}[lemma]{Lemma}
\aliascntresetthe{lemma}

\newaliascnt{proposition}{theorem}
\newtheorem{proposition}[proposition]{Proposition}
\aliascntresetthe{proposition}

\newaliascnt{corollary}{theorem}

\aliascntresetthe{corollary}

\theoremstyle{definition}
\newaliascnt{definition}{theorem}

\aliascntresetthe{definition}

\newaliascnt{example}{theorem}

\aliascntresetthe{example}

\theoremstyle{remark}
\newaliascnt{remark}{theorem}
\newtheorem{remark}[remark]{Remark}
\aliascntresetthe{remark}

\crefname{theorem}{theorem}{theorems}
\crefname{lemma}{lemma}{lemmas}
\crefname{proposition}{proposition}{propositions}
\crefname{corollary}{corollary}{corollaries}
\crefname{definition}{definition}{definitions}
\crefname{example}{example}{examples}
\crefname{remark}{remark}{remarks}

\newcommand{\R}{\mathbb R}
\newcommand{\Tcal}{\mathcal T}
\newcommand{\Jcal}{\mathcal J}
\newcommand{\dd}{\,\mathrm d}
\DeclareMathOperator{\Res}{Res}

\title{Real-rootedness of the $Z$-polynomials of sparse paving matroids
and strict interlacing for uniform matroids}
\author{Matthew H. Y. Xie\textsuperscript{1} and
Philip B. Zhang\textsuperscript{2}\\[6pt]
  \small \textsuperscript{1}School of Mathematical Sciences,\\
  \small Tianjin University of Technology, Tianjin 300384, P. R. China\\[4pt]
  \small \textsuperscript{2}College of Mathematical Sciences\\
  \small \& Institute of Mathematics and Interdisciplinary Sciences,\\
  \small Tianjin Normal University, Tianjin 300387, P. R. China\\[6pt]
  \small Email: \textsuperscript{1}\href{mailto:xie@email.tjut.edu.cn}{\texttt{xie@email.tjut.edu.cn}},
  \textsuperscript{2}\href{mailto:zhang@tjnu.edu.cn}{\texttt{zhang@tjnu.edu.cn}}}
\date{}

\hypersetup{
pdfauthor={Matthew H. Y. Xie and Philip B. Zhang},
pdftitle={Real-rootedness of the Z-polynomials of sparse paving matroids and strict interlacing for uniform matroids},
pdfsubject={Real-rootedness and interlacing of Z-polynomials and gamma-polynomials of matroids},
pdfkeywords={sparse paving matroid, Z-polynomial, real-rootedness, Jensen polynomial, Euler operator, strict interlacing}
}

\begin{document}
\maketitle

\begin{abstract}
We prove that the $Z$-polynomial of every sparse paving matroid has only
negative real zeros, confirming the real-rootedness conjecture of Proudfoot,
Xu, and Young for this class. We also prove that, for each fixed positive
corank, the $Z$-polynomials of uniform matroids in consecutive ranks strictly
interlace. Our approach is to prove real-rootedness of the corresponding
$\gamma$-polynomials for sparse paving matroids and their strict interlacing
for uniform matroids, and then transfer both properties to the corresponding
$Z$-polynomials.
\end{abstract}



\section{Introduction}
Let $M$ be a matroid with ground set $E(M)$ and lattice of flats $L(M)$.
Elias, Proudfoot, and Wakefield~\cite{EPW2016} introduced the Kazhdan--Lusztig
polynomial $P_M(x)$. Proudfoot, Xu, and Young~\cite{PXY2018} subsequently
defined the $Z$-polynomial by
\(
Z_M(x)=\sum_{F\in L(M)}x^{\operatorname{rk}F}P_{M/F}(x),
\)
and proved that it is palindromic of degree $\operatorname{rk}M$. They also
conjectured that all zeros of $Z_M(x)$ are real and negative and gave a
cohomological interpretation of $Z_M(x)$ for realizable matroids. By
developing a singular Hodge theory for arbitrary matroids, Braden, Huh,
Matherne, Proudfoot, and Wang~\cite{BHM2026} proved that $P_M(x)$ has
nonnegative coefficients. The defining sum above therefore shows that
$Z_M(x)$ also has nonnegative coefficients.

Recently, Cheng and Liu~\cite{ChengLiu2026} constructed representable matroids
with nonunimodal Kazhdan--Lusztig polynomials by deleting points from finite
projective geometries, disproving the general log-concavity and real-rootedness
conjectures for $P_M(x)$. They further showed that, under a condition relating
the degree to half the rank, such a deletion has the same $Z$-polynomial as the
ambient projective geometry. Since Proudfoot, Xu, and
Young~\cite{PXY2018} proved that this polynomial is real-rooted, these examples
do not contradict the real-rootedness conjecture for $Z_M(x)$.

We study the real-rootedness of $Z_M(x)$ through its $\gamma$-polynomial.
Write $d=\operatorname{rk}M$. Since $Z_M(x)$ is palindromic of degree $d$,
there is a unique polynomial $\Gamma_M(x)$ of degree at most
$\lfloor d/2\rfloor$ such that
\begin{equation}\label{eq:intro-gamma-transform}
Z_M(x)
=
(1+x)^d
\Gamma_M\!\left(\frac{x}{(1+x)^2}\right).
\end{equation}
Following Ferroni, Nasr, and Vecchi~\cite{FNV2023}, we call
$\Gamma_M(x)$ the $\gamma$-polynomial of $M$. They proved that every sparse
paving matroid is $\gamma$-positive, meaning that $\Gamma_M(x)$ has
nonnegative coefficients. Ferroni, Matherne, Stevens, and
Vecchi~\cite{FMSV2024} later extended $\gamma$-positivity to all matroids.
This strengthens the unimodality theorem of Braden, Huh, Matherne,
Proudfoot, and Wang~\cite{BHM2026}.

Before the present work, real-rootedness of $Z_M(x)$ had been established for
several special classes and bounded parameter ranges. Proudfoot, Xu, and
Young~\cite[Proposition~5.5]{PXY2018} proved real-rootedness, together with
strict interlacing, for the nice family associated with finite projective
geometries. For uniform matroids, the case of corank one is the classical
Narayana case, and Gao, Lu, Xie, Yang, and Zhang~\cite{GaoEtAl2021} proved
real-rootedness when the corank is between $2$ and $15$. Lu, Xie, and
Yang~\cite{LXY2022} proved real-rootedness for the $Z$-polynomials of fan
matroids, wheel matroids, and whirl matroids. Ferroni and
Vecchi~\cite{FerroniVecchi2022} verified real-rootedness for every sparse
paving matroid on at most $30$ elements. Recently, Zhang~\cite{Zhang2026}
proved that, for $n\ge2$, the common $Z$-polynomial of the graphic matroids of
$K_{1,1,n}$ and $K_{2,n}$ has $n+1$ distinct negative zeros.

A particularly natural class to consider is that of sparse paving matroids.
This class contains all uniform matroids, and it is conjectured that
asymptotically almost all matroids are sparse
paving~\cite{MayhewEtAl2011,PendavinghVanderPol2015}. A matroid is paving if
every circuit has
cardinality at least its rank, and it is sparse paving if both the matroid and
its dual are paving. A circuit-hyperplane is a circuit that is also a
hyperplane.
Our first main result proves that the real-rootedness conjecture for
$Z$-polynomials holds for every sparse paving matroid.

\begin{theorem}\label{thm:main}
Let $M$ be a sparse paving matroid. Then both $\Gamma_M(x)$ and $Z_M(x)$
have only negative real zeros.
\end{theorem}

The proof reduces the problem to the real-rootedness of the Poisson transform
of an auxiliary polynomial. Using a Laguerre partial-fraction expansion, we
control its poles and residues and thereby locate all of its zeros on the real
axis. Jensen's theorem and the finite $n$-sequence criterion then yield the
real-rootedness of $\Gamma_M(x)$, while
\eqref{eq:intro-gamma-transform} transfers the result to $Z_M(x)$.

Uniform matroids form a subclass of sparse paving matroids. Our second main
result proves that, for uniform matroids of fixed corank, the $Z$-polynomials
in consecutive ranks strictly interlace. This strengthens the weak interlacing
predicted by the conjecture of Proudfoot, Xu, and Young~\cite{PXY2018}.
Recall that they call a sequence
$\{M_d\}_{d\ge0}$ a \emph{nice family} if $\operatorname{rk}M_d=d$ and, for
every flat $F$ of $M_d$ of corank $k$, the contraction $M_d/F$ has a lattice
of flats isomorphic to that of $M_k$. They conjectured that $Z_{M_d}(x)$
interlaces $Z_{M_{d-1}}(x)$ for every $d\ge1$ in every nice family. Together
with their real-rootedness conjecture, this predicts interlacing on the
negative real axis.

We write $P\preccurlyeq Q$ and $P\prec Q$ for weak and strict interlacing in
the direction specified in \Cref{sec:interlacing-criteria}. With this notation,
the conjecture of Proudfoot, Xu, and Young reads
$Z_{M_{d-1}}(x)\preccurlyeq Z_{M_d}(x)$ for every $d\ge1$.

Following Proudfoot, Xu, and Young~\cite{PXY2018}, for $m,d\ge0$ we write
$U_{m,d}$ for the uniform matroid of rank $d$ and corank $m$, that is, the
uniform matroid of rank $d$ on $m+d$ elements. For each fixed $m\ge1$, the
matroids $\{U_{m,d}\}_{d\ge0}$ form a nice family. We prove the following
strict form of the conjectured comparison.

\begin{theorem}\label{thm:uniform-Z-strict-interlacing}
For $m,d\ge1$,
\(
Z_{U_{m,d}}(x)\prec Z_{U_{m,d+1}}(x).
\)
\end{theorem}

The associated $\gamma$-polynomials satisfy the following three strict
interlacing relations.

\begin{theorem}\label{thm:uniform-gamma-all-interlacings}
(i) For $m,d\ge1$,
\(
\Gamma_{U_{m,d}}(x)\prec\Gamma_{U_{m,d+1}}(x)
\).

(ii) For $m\ge1$ and $d\ge2$,
\(
\Gamma_{U_{m,d}}(x)\prec\Gamma_{U_{m+1,d}}(x)
\).

(iii) For $m,d\ge1$,
\(
\Gamma_{U_{m,d}}(x)\prec\Gamma_{U_{m+1,d+1}}(x)
\).
\end{theorem}
For $m\ge1$ and $d\ge2$, the five strict interlacing relations shown below
hold in the direction of the arrows:
\[
\begin{array}{c@{\qquad}c@{\qquad}c}
\Gamma_{U_{m,d}}(x)
& \longrightarrow
& \Gamma_{U_{m,d+1}}(x)
\\[1mm]
\downarrow & \searrow & \downarrow
\\[1mm]
\Gamma_{U_{m+1,d}}(x)
& \longrightarrow
& \Gamma_{U_{m+1,d+1}}(x).
\end{array}
\]

To prove the interlacing results for uniform matroids, we rewrite the
coefficient formulas using the Euler operator
$\vartheta=x\,\mathrm d/\mathrm dx$. For each of the three comparisons between
the $\gamma$-polynomials, the rank and corank recursions allow us to write the
two polynomials as $A(\vartheta)H(x)$ and $B(\vartheta)H(x)$ with a common polynomial $H(x)$. A
connection with Jacobi polynomials shows that $H(x)$ has distinct negative zeros.
We compare the zeros of $A$ and $B$ and then apply a criterion for two Euler
operators acting on the same polynomial.

The rest of this paper is organized as follows. In
\Cref{sec:preliminaries}, we derive the Euler-operator formulas and establish
the results on $n$-sequences, Jensen polynomials, and interlacing needed later.
\Cref{sec:real-rootedness} proves \Cref{thm:main} by establishing the
real-rootedness of the Poisson transform from its residue expansion and then
applying Jensen's theorem and the criterion for $n$-sequences. The cases
$d=4$ and $d=5$ follow from direct evaluations of the $\gamma$-polynomials. In
\Cref{sec:uniform-interlacing}, we locate the zeros of $N_{m,d}(y)$, prove the
three strict comparisons in \Cref{thm:uniform-gamma-all-interlacings}, and
deduce \Cref{thm:uniform-Z-strict-interlacing}.

\section{The Euler operator, \texorpdfstring{$n$}{n}-sequences, and interlacing}
\label{sec:preliminaries}

Throughout this section, $\vartheta=x\,\mathrm d/\mathrm dx$ denotes the
Euler operator.

\subsection{The \texorpdfstring{$\gamma$}{gamma}-polynomials of sparse paving
matroids and the Euler operator}

Since $Z_M(x)$ is palindromic and $Z_M(0)=1$, its
unimodality~\cite[Theorem~1.2]{BHM2026} implies that every coefficient is
positive. Consequently, the equivalence in
\cite[Proposition~5.3]{FNV2023} shows that $Z_M(x)$ is real-rooted if and only
if $\Gamma_M(x)$ has only negative real zeros. Thus the real-rootedness assertion
in \Cref{thm:main} reduces to proving that $\Gamma_M(x)$ has only negative real
zeros. For a sparse paving matroid, we
express $\Gamma_M(x)$ as a polynomial in the Euler operator acting on an
auxiliary polynomial that will be shown to have distinct negative zeros. This
Euler-operator expression is the starting point for the proof in
\Cref{sec:real-rootedness}.

For positive integers $m$ and $d$, write
$\Gamma_{m,d}(x)=\Gamma_{U_{m,d}}(x)$.
For $j\ge0$, let $(y)_0=1$ and
$(y)_j=y(y+1)\cdots(y+j-1)$ for $j\ge1$, and define
\begin{equation}\label{eq:N-uniform-def}
N_{m,d}(y)
=
\sum_{j=0}^{m-1}
(m-j)!\binom{m+d}{j}(y)_j.
\end{equation}
For $y\ge0$, every summand is nonnegative and the term with $j=0$ is $m!$.
Hence $N_{m,d}(y)\ge m!>0$. We first express $\Gamma_{m,d}(x)$ in terms of
$N_{m,d}(y)$.

\begin{proposition}
\label{prop:uniform-Gamma-coefficients}
For positive integers $m$ and $d$,
\begin{equation}
\Gamma_{m,d}(x)
=
\sum_{i=0}^{\lfloor d/2\rfloor}
\frac{d!}{(d-2i)!\,i!\,(i+m)!}
N_{m,d}(i)x^i.
\label{eq:uniform-coeff}
\end{equation}
\end{proposition}

\begin{proof}
For $i=0$, the coefficient on the right is $1$, which is the constant term of
$\Gamma_{m,d}(x)$. Assume that $1\le i\le\lfloor d/2\rfloor$.
        We use the following formula of Wu,
Xie, and Zhang~\cite[Theorem~3, Eq.~(2)]{WXZ2024}:
\[
        [x^i]\Gamma_{m,d}(x)
=
\frac{d!}{(d-2i)!\,i!\,(i+m)!}S_{m,d}(i),
\]
where
\(
S_{m,d}(i)
=\frac{(i+m)!}{(i-1)!}
\sum_{h=0}^{m-1}
\frac{\binom{d+h}{h}}{(h+i)(h+i+1)}.
\)
Then $S_{1,d}(i)=1$, and separating the last summand gives
\[
S_{m+1,d}(i)
=(i+m+1)S_{m,d}(i)
+\binom{m+d}{m}(i)_m.
\]
On the other hand, the identity
$(y+m+1)(y)_j=(y)_{j+1}+(m+1-j)(y)_j$ and Pascal's identity give
\begin{equation}\label{eq:N-recursion}
N_{m+1,d}(y)
=(y+m+1)N_{m,d}(y)
+\binom{m+d}{m}(y)_m.
\end{equation}
Since $N_{1,d}(y)=1$, induction on $m$ yields
$S_{m,d}(i)=N_{m,d}(i)$, proving \eqref{eq:uniform-coeff}.
\end{proof}

For positive integers $m$ and $d$ and every $\lambda\in\R$, define
\begin{equation}
\Gamma_{m,d,\lambda}(x)
=\Gamma_{m,d}(x)
-\lambda\sum_{i=1}^{\lfloor d/2\rfloor}
\frac{2}{i+1}
\binom{d-1}{i-1}
\binom{d-i}{i}x^i.
\label{eq:sparse-gamma-formula}
\end{equation}
If $M$ is a sparse paving matroid of rank $d$ and corank $m$ with $\lambda$
circuit-hyperplanes, then the formula of Ferroni, Nasr, and
Vecchi~\cite[Eq.~(12), Proposition~5.12, and Remark~5.14]{FNV2023} gives
\(
\Gamma_M(x)=\Gamma_{m,d,\lambda}(x).
\)
Thus the $\gamma$-polynomial of $M$ depends only on $m$, $d$, and $\lambda$.
In particular, $\Gamma_{m,d,0}(x)=\Gamma_{m,d}(x)$.

For $m,d\ge1$ and $\lambda\in\R$, define
\begin{equation}\label{eq:N-def}
N_{m,d,\lambda}(y)
=
N_{m,d}(y)
-\frac{2\lambda}{d}\,
y(y+2)_{m-1}.
\end{equation}
Thus $N_{m,d,0}(y)=N_{m,d}(y)$. We also set
\[
f_{m,d,i}
=
\frac{d!}{(d-2i)!\,i!\,(i+m)!}
\quad
\left(0\le i\le\left\lfloor\frac d2\right\rfloor\right).
\]

Multiplying the correction term in $N_{m,d,\lambda}(i)$ by $f_{m,d,i}$ gives
exactly the coefficient correction in \eqref{eq:sparse-gamma-formula}, so
\Cref{prop:uniform-Gamma-coefficients} yields the following identity.

\begin{proposition}
\label{prop:Gamma-coefficients}
Let $m$ and $d$ be positive integers, and let $\lambda\in\R$. Then, for every
integer $i$ with $0\le i\le\lfloor d/2\rfloor$,
\begin{equation}\label{eq:Gamma-coefficients}
[x^i]\Gamma_{m,d,\lambda}(x)
=f_{m,d,i}N_{m,d,\lambda}(i).
\end{equation}
\end{proposition}

Define the polynomial
\[
H_{m,d}(x)
:=
\sum_{i=0}^{\lfloor d/2\rfloor}f_{m,d,i}x^i.
\]
Since $\vartheta x^i=ix^i$, \Cref{prop:Gamma-coefficients} yields
\begin{equation}\label{eq:Gamma-Euler}
\Gamma_{m,d,\lambda}(x)
=N_{m,d,\lambda}(\vartheta)H_{m,d}(x).
\end{equation}

\begin{lemma}
\label{lem:H-simple-roots}
For every $m,d\ge1$, the polynomial $H_{m,d}(x)$ has distinct negative real
zeros.
\end{lemma}

\begin{proof}
With ${}_2F_1$ denoting the Gauss hypergeometric function, we have
\[
H_{m,d}(x)
=\frac1{m!}
{}_2F_1\!\left(-\frac d2,\frac{1-d}{2};m+1;4x\right).
\]
Write $d=2q+\delta$, where $q\ge0$ and $\delta\in\{0,1\}$. The
hypergeometric form of the Jacobi polynomials~\cite[Eq.~18.5.7]{DLMF},
together with Pfaff's
transformation~\cite[Eq.~15.8.1]{DLMF}, shows that the zeros of
$H_{m,2q+\delta}(x)$ in $(-\infty,0)$ are precisely the inverse images under
$x\mapsto(1+4x)/(1-4x)$ of the zeros of
$P_q^{(m,\delta-\frac12)}$. The factor relating the two polynomials is
$q!(1-4x)^q/(m+q)!$, which does not vanish on $(-\infty,0)$. Since
$m>-1$ and $\delta-\frac12>-1$, the Jacobi polynomial
$P_q^{(m,\delta-\frac12)}$ has $q$ distinct zeros in
$(-1,1)$~\cite[Table~18.3.1 and Section~18.2(vi)]{DLMF}. The map above is
increasing from $(-\infty,0)$ onto $(-1,1)$.
Thus $H_{m,d}(x)$ has $q$ distinct negative zeros. Since
$\deg H_{m,d}(x)=q$, these are all of its zeros.
\end{proof}

The following recursions compare $H_{m,d}(x)$ as the rank or the corank
varies.

\begin{lemma}\label{lem:H-recursions}
For $m,d\ge1$,
\begin{align}
H_{m,d}(x)
&=\left(1-\frac{2\vartheta}{d+1}\right)H_{m,d+1}(x),
\label{eq:H-rank-recursion}\\
H_{m,d}(x)
&=(\vartheta+m+1)H_{m+1,d}(x).
\label{eq:H-corank-recursion}
\end{align}
\end{lemma}

\begin{proof}
For every $0\le i\le\lfloor d/2\rfloor$,
\(
\frac{f_{m,d,i}}{f_{m,d+1,i}}
=\frac{d+1-2i}{d+1},
\frac{f_{m,d,i}}{f_{m+1,d,i}}
=i+m+1.
\)
Since $\vartheta x^i=ix^i$, these ratios give the coefficients in both
identities. If $d$ is odd, then $H_{m,d+1}(x)$ has the
additional term indexed by $i=(d+1)/2$, but its multiplier in the first
identity is zero.
\end{proof}

\subsection{\texorpdfstring{$n$}{n}-sequences and Jensen polynomials}

By \Cref{lem:H-simple-roots}, the polynomial $H_{m,d}(x)$ in
\eqref{eq:Gamma-Euler} is real-rooted. Thus, to prove that
$\Gamma_{m,d,\lambda}(x)$ is real-rooted, it is enough to show that
$N_{m,d,\lambda}(\vartheta)$ preserves real-rootedness. Since
$\vartheta x^i=ix^i$, every operator $A(\vartheta)$ is diagonal in the
monomial basis.

For $A\in\R[y]$ and $n\ge0$, define
\(
\Jcal_n[A](x)
=A(\vartheta)(1+x)^n
=\sum_{i=0}^{n}\binom ni A(i)x^i.
\)
This is the $n$th Jensen polynomial, with shift $0$, associated with
$A(0),A(1),\ldots$; the construction goes back to
Jensen~\cite[p.~184]{Jensen1913}. Following Craven and
Csordas~\cite[Definition~1.11]{CravenCsordas1983}, and allowing the zero
polynomial as an output, we call a real sequence $(\mu_0,\ldots,\mu_n)$ an
\emph{$n$-sequence of the first kind} if, for every real-rooted polynomial
\(
F(x)=\sum_{i=0}^{n}a_ix^i,
\)
the polynomial $\sum_{i=0}^{n}\mu_i a_i x^i$ is either the zero polynomial or
has only real zeros.
When $\mu_i=A(i)$, this coefficientwise action is precisely
$A(\vartheta)F(x)$. Thus $A(0),\ldots,A(n)$ form an $n$-sequence of the first
kind exactly when $A(\vartheta)$ preserves real-rootedness on polynomials of
degree at most $n$. In this paper, we use
``finite multiplier sequence of the first kind'' as an equivalent name for
such an $n$-sequence.
     Specializing the original finite characterization of Craven and Csordas
to $\R$ gives the following criterion.

\begin{theorem}[{\cite[Theorem~3.7]{CravenCsordas1977}}]
\label{thm:finite-multiplier}
Let $n\ge1$ and $A\in\R[y]$, and suppose that
$A(0),\ldots,A(n)$ are not all zero. Then
$A(0),\ldots,A(n)$ form an $n$-sequence of the first kind if and only if all
zeros of $\Jcal_n[A](x)$ are real and have the same sign.
\end{theorem}

Jensen's theorem produces real-rooted test polynomials from exponential
generating functions.

\begin{theorem}[{\cite[pp.~183--187]{Jensen1913}}]
\label{thm:Jensen}
Suppose that
\(
\Phi(x)=e^{bx}P(x)
=
\sum_{\ell\ge0}
\gamma_\ell\frac{x^\ell}{\ell!},
\)
where $b\in\R$ and $P\in\R[x]$ is a nonzero polynomial having only real
zeros. Then, for every $n\ge0$, the polynomial
\(
\sum_{i=0}^{n}\binom ni\gamma_i x^i
\)
is either identically zero or has only real zeros.
\end{theorem}

\subsection{Interlacing criteria}
\label{sec:interlacing-criteria}

We first fix the convention for the direction of interlacing. Let $P$ and $Q$ be
nonzero real-rooted polynomials with positive leading coefficients and degrees
$q$ and $r \in \{q, q+1\}$, respectively. Denote their roots (counting multiplicity)
by $a_1\le\cdots\le a_q$ and $b_1\le\cdots\le b_r$. We write $P\preccurlyeq Q$ if
\(
a_1\le b_1\le\cdots\le a_q\le b_q
\)
when $r=q$, or
\(
b_1\le a_1\le\cdots\le b_q\le a_q\le b_{q+1}
\)
when $r=q+1$. When $q=0$, these correspond to the empty chain and the
single-term chain $b_1$, respectively. Finally, we write $P\prec Q$ if
$P\preccurlyeq Q$, all holding inequalities are strict, and $P$ and $Q$ share no common roots.

We use the Hermite--Kakeya--Obreschkoff theorem.

\begin{theorem}[{\cite[Theorem~1.9]{BorceaBranden2009I};
\cite[Theorem~6.3.8]{RahmanSchmeisser2002}}]
\label{thm:HKO}
Let $P(x),Q(x)\in\R[x]$ be linearly independent nonconstant polynomials.
Every nonzero polynomial in $\operatorname{span}_{\R}\{P,Q\}$ is real-rooted
if and only if $P$ and $Q$ are real-rooted and their zeros interlace.
Moreover, every nonzero polynomial in this span has only simple real zeros if
and only if $P$ and $Q$ have simple, strictly interlacing zeros.
\end{theorem}

For strict interlacing, the basic step is the action of one linear factor on a
polynomial with negative zeros.

\begin{lemma}\label{lem:Euler-one-factor}
Let $R$ have degree $n\ge1$ and distinct negative zeros. For every $r\in\R$,
the polynomial $(\vartheta-r)R$ is nonzero and has only simple real zeros. If
$r\notin[0,n]$, it has degree $n$ and all its zeros are negative.
\end{lemma}

\begin{proof}
Write the zeros of $R$ as $\rho_1<\cdots<\rho_n<0$ and put
$P=(\vartheta-r)R$. If $r=0$, then $P=xR'$, whose zeros are simple by the
strict interlacing of $R'$ and $R$; the zero at the origin is simple because
\(
\frac{R'(0)}{R(0)}
=\sum_{\nu=1}^{n}\frac1{-\rho_\nu}>0.
\)
Assume that $r\ne0$. Since
$P(\rho_j)=\rho_jR'(\rho_j)\ne0$, its zeros are the solutions of
$\Phi(x)=r$, where
\(
\Phi(x)=\frac{xR'(x)}{R(x)}
=\sum_{\nu=1}^{n}\frac{x}{x-\rho_\nu},
\Phi'(x)=\sum_{\nu=1}^{n}
\frac{-\rho_\nu}{(x-\rho_\nu)^2}>0.
\)
The function $\Phi$ maps each $(\rho_j,\rho_{j+1})$ onto $\R$, giving
$n-1$ simple negative zeros. Its ranges on
$(-\infty,\rho_1)$, $(\rho_n,0)$, and $(0,\infty)$ are
$(n,\infty)$, $(-\infty,0)$, and $(0,n)$, respectively. These ranges give an
additional real zero unless $r=n$; in that case the leading term of $P$
vanishes and the $n-1$ zeros already found exhaust its degree. Every zero
$\xi$ is simple because $P'(\xi)=R(\xi)\Phi'(\xi)\ne0$. In particular, the
additional zero is negative when $r<0$ or $r>n$.
\end{proof}

The Wronskian detects weak interlacing. In the strict case, its fixed sign makes
the quotient strictly monotone, so each nonzero linear combination has at most
one zero on any interval containing no zero of the denominator.

\begin{lemma}
\label{lem:strict-interlacing-quotient}
Let $P$ and $Q$ be nonconstant real-rooted polynomials with positive leading
coefficients and weakly interlacing zeros, in either order. Then their
Wronskian $W[P,Q]=P'Q-PQ'$ is either nonnegative on $\R$ or nonpositive on
$\R$. If their zeros are
distinct and strictly interlace, then the Wronskian is either positive on
$\R$ or negative on $\R$. In the strict case, $P/Q$ is strictly monotone on
every interval containing no zero of $Q$. In particular, every nonzero linear
combination of $P$ and $Q$ has at most one zero in such an interval.
\end{lemma}

\begin{proof}
The Wronskian form of the Hermite--Biehler theorem
\cite[Theorem~6.3.4]{RahmanSchmeisser2002} shows that $W[P,Q]$ is nonnegative
everywhere or nonpositive everywhere. In the strict case, a zero $x_0$ of the
Wronskian would make the nonzero polynomial
$Q(x_0)P-P(x_0)Q$ have a double zero at $x_0$, contrary to the strict part of
\Cref{thm:HKO}. Thus the Wronskian never vanishes, and
$(P/Q)'=W[P,Q]/Q^2$ has a fixed nonzero sign away from the zeros of $Q$.
If $uP+vQ$ has two zeros in such an interval and $u\ne0$, then
$P/Q=-v/u$ at two distinct points, a contradiction. If $u=0$, the combination
is a nonzero multiple of $Q$ and has no zero in the interval.
\end{proof}

The action of a linear factor and quotient monotonicity combine to give the
criterion used when two Euler operators act on the same polynomial $H(x)$.

\begin{lemma}\label{lem:common-H-interlacing}
Let $H(x)$ have degree $q\ge1$ and distinct negative zeros. Let
$A,B\in\R[y]$ be linearly independent polynomials with simple strictly
interlacing real zeros.
Suppose that $B$ has no zero in $[0,q]$ and that
$A(\vartheta)H(x)$ and $B(\vartheta)H(x)$ are nonconstant. Then
\(
A(\vartheta)H(x)
\text{ and }
B(\vartheta)H(x)
\)
have simple strictly interlacing real zeros.
\end{lemma}

\begin{proof}
We first show that every nonzero polynomial in
$\operatorname{span}_{\R}\{A,B\}$ has simple real zeros and at most one zero
in $[0,q]$. If one of $A$ and $B$ is constant, strict interlacing and linear
independence force the other to have degree one, and both assertions are
immediate. Otherwise, \Cref{thm:HKO} shows that every nonzero polynomial in
the span has simple real zeros. Choose
$\varepsilon_A,\varepsilon_B\in\{-1,1\}$ so that $\varepsilon_AA$ and
$\varepsilon_BB$ have positive leading coefficients. Their zeros still
strictly interlace. Since $B$ has no zero in $[0,q]$,
\Cref{lem:strict-interlacing-quotient} shows that
$(\varepsilon_AA)/(\varepsilon_BB)$ is strictly monotone on $[0,q]$. This
quotient differs from $A/B$ by a nonzero constant factor, so $A/B$ is also
strictly monotone there. Hence every nonzero polynomial in the span has at
most one zero in $[0,q]$.

Fix a nonzero $C\in\operatorname{span}_{\R}\{A,B\}$ and write
$C(y)=c\prod_{\nu=1}^{s}(y-r_\nu)$. If $s=0$, then
$C(\vartheta)H(x)=cH(x)$ already has distinct negative zeros, so we may assume that
$s\ge 1$. Its roots are real and distinct. Since
the operators $\vartheta-r_\nu$ commute, apply first the factors with
$r_\nu\notin[0,q]$. By \Cref{lem:Euler-one-factor}, every such factor
preserves the degree $q$ and leaves distinct negative zeros. At most one
factor remains, and the same lemma shows that its application, if needed,
produces a nonzero polynomial with simple real zeros. Thus every nonzero
$C(\vartheta)H(x)$ has simple real zeros.

The map $C\mapsto C(\vartheta)H(x)$ is therefore injective on
$\operatorname{span}_{\R}\{A,B\}$. Its images of $A$ and $B$ are linearly
independent, and every nonzero polynomial in their span has simple real zeros.
The strict part of \Cref{thm:HKO} therefore gives the required interlacing.
\end{proof}

When two resulting polynomials have the same degree, the following elementary
criterion determines their order.

\begin{lemma}\label{lem:interlacing-direction}
Let $P$ and $Q$ have the same positive degree, constant term $1$, and simple
negative zeros. Suppose that their zeros strictly interlace. Then
\(
P\prec Q
        \)
if and only if
        \(
        P'(0)<Q'(0).
\)
\end{lemma}

\begin{proof}
Write the zeros as
\(
a_1<\cdots<a_q<0,
b_1<\cdots<b_q<0.
\)
Since $P(0)=Q(0)=1$,
\(
P'(0)=\sum_{i=1}^{q}\frac1{-a_i},
Q'(0)=\sum_{i=1}^{q}\frac1{-b_i}.
\)
If $P\prec Q$, then $a_i<b_i$ for every $i$, so $P'(0)<Q'(0)$. The reverse
interlacing order gives the reverse inequality.
\end{proof}

The strict interlacing of the $\gamma$-polynomials will be transferred to the
$Z$-polynomials through a palindromic transformation. Hoster and Stump proved
that palindromic polynomials of consecutive degrees with nonnegative
coefficients weakly interlace if and only if their $\gamma$-polynomials do
\cite[Proposition~2.5]{HosterStump2025}.

\begin{proposition}
\label{prop:transformation-preserves-interlacing}
Let $k\ge2$, and let $G_k,G_{k+1}$ have positive leading coefficients and
distinct negative real zeros. Suppose
$\deg G_j=\lfloor j/2\rfloor$ for $j=k,k+1$. Define
\(
Q_j(x)=(1+x)^jG_j\!\left(\frac{x}{(1+x)^2}\right).
\)
Then each $Q_j$ has $j$ distinct negative real zeros.
        Furthermore, 
\(
G_k(x)\prec G_{k+1}(x)
        \)
if and only if 
        \(
        Q_k(x)\prec Q_{k+1}(x).
\)
\end{proposition}

\begin{proof}
Each $G_j$ has positive coefficients. Hence
\(
Q_j(x)=\sum_i [y^i]G_j(y)\,x^i(1+x)^{j-2i}
\)
is palindromic of degree $j$, belongs to $\R_{\ge0}[x]$, and has
$\gamma$-polynomial $G_j(x)$. The cited proposition therefore gives the weak
equivalence.

Every zero $a<0$ of $G_j$ lifts through $x/(1+x)^2=a$ to two distinct
reciprocal negative zeros, neither equal to $-1$. The degree assumption on
$G_j$ shows that these are all zeros of $Q_j$ when $j$ is even; when $j$ is
odd, $Q_j$ has the additional simple zero $-1$. Thus every $Q_j$ has $j$
distinct negative zeros. Exactly one of $Q_k,Q_{k+1}$ has the zero $-1$.
Every common zero of $Q_k$ and $Q_{k+1}$ therefore descends under
$x\mapsto x/(1+x)^2$, and each common zero of $G_k$ and $G_{k+1}$ lifts to
common zeros of $Q_k$ and $Q_{k+1}$. Hence one pair has a common zero if and
only if the other does. For polynomials with simple zeros, weak interlacing is
strict exactly when the two polynomials have no common zero. Combining this
criterion with the equivalence for weak interlacing proves the result.
\end{proof}

\section{Real-rootedness of the \texorpdfstring{$Z$}{Z}-polynomials of
sparse paving matroids}
\label{sec:real-rootedness}

Let $M$ be a sparse paving matroid of rank $d\ge1$ and positive corank $m$,
with $\lambda$ circuit-hyperplanes. By \eqref{eq:Gamma-Euler}, its
$\gamma$-polynomial is $N_{m,d,\lambda}(\vartheta)H_{m,d}(x)$. Since
$H_{m,d}(x)$ already has distinct negative zeros, it suffices to show that
$N_{m,d,\lambda}(\vartheta)$ preserves real-rootedness in the required degree.
For $d\ge6$, rather than study this operator directly, we prove that the
Poisson transform of $N_{m,d,\lambda}(y)$ is real-rooted. The real-rootedness
of this transform, together with $\gamma$-positivity, Jensen's theorem, and the
criterion for $n$-sequences, gives the required preservation property.
Equation~\eqref{eq:Gamma-Euler} gives real-rootedness of
$\Gamma_M(x)$, and \eqref{eq:intro-gamma-transform} gives real-rootedness of
$Z_M(x)$. If $d\le3$, then $\Gamma_M(x)$ has degree at most one, while
\Cref{lem:rank45-discriminants} proves that $\Gamma_M(x)$ is real-rooted when
$d=4$ or $5$. In these cases, $\gamma$-positivity and
\eqref{eq:intro-gamma-transform} show that both $\Gamma_M(x)$ and $Z_M(x)$
have only negative real zeros.

\subsection{The Poisson transform and Laguerre polynomials}

For a polynomial $P(y)$, we define its Poisson transform~\cite{JacquetSzpankowski1998} by
\begin{equation}\label{eq:Poisson-def}
\Tcal[P](x)
=
e^{-x}\sum_{\ell\ge0}
P(\ell)\frac{x^\ell}{\ell!}
=
e^{-x}
P\!\left(x\frac{\dd}{\dd x}\right)e^x.
\end{equation}
For the polynomial $N_{m,d,\lambda}(y)$ in \eqref{eq:N-def}, write
\(
T_{m,d,\lambda}(x)=\Tcal[N_{m,d,\lambda}](x).
\)

Multiplying the first equality in \eqref{eq:Poisson-def} by $e^x$ gives
\begin{equation}\label{eq:Poisson-egf}
\sum_{\ell\ge0}
P(\ell)\frac{x^\ell}{\ell!}
=
e^x\Tcal[P](x).
\end{equation}
Since $P(\ell)$ has polynomial growth, the series in \eqref{eq:Poisson-egf}
converges locally uniformly on $\mathbb C$, as do its termwise derivatives.

For $k\ge0$, write
$y^{\underline{k}}=y(y-1)\cdots(y-k+1)$, with
$y^{\underline{0}}=1$. Reindexing the series in \eqref{eq:Poisson-def} shows
that $\Tcal[y^{\underline{k}}](x)=x^k$. Since the falling factorials form a
monic basis of $\R[y]$, the Poisson transform preserves degree and leading
coefficient.

The rising factorials in \eqref{eq:N-def} have particularly simple Poisson
transforms in terms of generalized Laguerre polynomials. For
$n\in\mathbb Z_{\ge0}$ and $\alpha\ge0$, we use the normalization
\(
L_n^{(\alpha)}(x)
=\sum_{k=0}^{n}(-1)^k
\binom{n+\alpha}{n-k}\frac{x^k}{k!}
\)
for the generalized Laguerre polynomials; see~\cite[Eq.~18.5.12]{DLMF}.

\begin{lemma}\label{lem:Poisson-Laguerre}
For $n,\alpha\in\mathbb Z_{\ge0}$,
\(
\Tcal\bigl[y(y+\alpha)_n\bigr](x)
=n!xL_n^{(\alpha)}(-x).
\)
\end{lemma}

\begin{proof}
For every $Q\in\R[y]$, reindexing the defining series gives
\begin{align*}
\Tcal[yQ(y)](x)
&=e^{-x}\sum_{\ell\ge1}
\ell Q(\ell)\frac{x^\ell}{\ell!}\\
&=xe^{-x}\sum_{r\ge0}Q(r+1)\frac{x^r}{r!}\\
&=x\Tcal[Q(y+1)](x).
\end{align*}
Taking $Q(y)=(y+\alpha)_n$ and using Vandermonde's identity, we obtain
\begin{align*}
\Tcal\bigl[y(y+\alpha)_n\bigr](x)
&=x\Tcal[(y+\alpha+1)_n](x)\\
&=n!x\sum_{k=0}^{n}
\binom{n+\alpha}{n-k}\frac{x^k}{k!}\\
&=n!xL_n^{(\alpha)}(-x)
\end{align*}
as desired.
\end{proof}

Applying \Cref{lem:Poisson-Laguerre} to the terms in
\eqref{eq:N-uniform-def}, we obtain
\(
T_{m,d,0}(x)
=
m!
+
\sum_{j=1}^{m-1}
(m-j)!\binom{m+d}{j}(j-1)!xL_{j-1}^{(1)}(-x).
\)
Taking $(n,\alpha)=(m-1,2)$ in the identity for
$\Tcal[y(y+\alpha)_n]$ evaluates the correction term in
\eqref{eq:N-def}.  Hence
\begin{equation}\label{eq:T-lambda}
T_{m,d,\lambda}(x)
=
T_{m,d,0}(x)
-
\frac{2\lambda}{d}
(m-1)!xL_{m-1}^{(2)}(-x).
\end{equation}
For $\lambda>0$, the polynomial $T_{m,d,\lambda}(x)$ has degree $m$ and
leading term $-(2\lambda/d)x^m$, whereas $T_{m,d,0}(x)$ has degree $m-1$.

\subsection{Residues of normalized Poisson transforms}

Equation~\eqref{eq:T-lambda} writes $T_{m,d,\lambda}(x)$ as a difference
involving Laguerre polynomials, but it does not determine whether
$T_{m,d,\lambda}(x)$ is real-rooted. The main new device in the proof is to
divide $T_{m,d,\lambda}(x)$ by
$\Tcal[(y)_m](x)$ and count its zeros by determining and comparing the poles
and residues of the resulting rational function. To obtain the required
partial fraction expansion for $T_{m,d,0}(x)$, we express $N_{m,d}(y)$ in
terms of the polynomials obtained by deleting one factor from $(y)_m$.

For $m\ge2$ and $d\ge1$, recall that
\(
N_{m,d}(y)
=
\sum_{j=0}^{m-1}
(m-j)!\binom{m+d}{j}(y)_j.
\)
For $0\le j\le m-1$, set
\(
R_j(y)
=
\prod_{\substack{0\le h\le m-1\\h\ne j}}(y+h)
=
\frac{(y)_m}{y+j}.
\)

\begin{proposition}
\label{prop:N-R-expansion}
Let $m\ge2$ and $d\ge1$. Then
\begin{equation}\label{eq:N-R}
N_{m,d}(y)
=
\sum_{j=0}^{m-1}
(m-j)\frac{(d)_j}{j!}R_j(y).
\end{equation}
\end{proposition}

\begin{proof}
For $0\le k,j\le m-1$,
\[
R_j(-k)
=
\begin{cases}
(-1)^j j!(m-1-j)!,&k=j,\\
0,&k\ne j.
\end{cases}
\]
Thus $R_0,\ldots,R_{m-1}$ form a basis for the polynomials of degree at most
$m-1$. Since $(-k)_j=(-1)^j j!\binom{k}{j}$, the Chu--Vandermonde identity
gives
\begin{align*}
N_{m,d}(-k)
&=
m!\sum_{j=0}^k
\frac{(-k)_j(-m-d)_j}{(-m)_j j!}\\
&=
m!\frac{(d)_k}{(-m)_k}
=
(-1)^k(m-k)!(d)_k.
\end{align*}
Consequently,
\(
\frac{N_{m,d}(-k)}{R_k(-k)}
=
(m-k)\frac{(d)_k}{k!}.
\)
The coefficients in \eqref{eq:N-R} are therefore the interpolation
coefficients of $N_{m,d}(y)$ in this basis.
\end{proof}

By \Cref{lem:Poisson-Laguerre},
\(
\Tcal[(y)_m](x)=(m-1)!xL_{m-1}^{(1)}(-x).
\)
Let
\(
0<x_1<x_2<\cdots<x_{m-1},
\)
be the positive simple zeros of $L_{m-1}^{(1)}(x)$~\cite[Section~18.16(iv)]{DLMF}.

For $0\le j\le m-1$ and $1\le i\le m-1$, define
\(
b_{i,j}
:=
\Res_{x=-x_i}\frac{\Tcal[R_j](x)}{\Tcal[(y)_m](x)}
=
\frac{\Tcal[R_j](-x_i)}
{(m-1)!x_i\bigl(L_{m-1}^{(1)}\bigr)'(x_i)}.
\)

For $1\le j\le m-1$, canceling the common factor $x$ in
$\Tcal[R_j](x)/\Tcal[(y)_m](x)$ leaves a proper rational function whose only
possible poles are $-x_1,\ldots,-x_{m-1}$, all of order at most one. Hence
\begin{equation}\label{eq:R-partial-fraction}
\frac{\Tcal[R_j](x)}
{(m-1)!xL_{m-1}^{(1)}(-x)}
=
\sum_{i=1}^{m-1}
\frac{b_{i,j}}{x+x_i}.
\end{equation}
For $j=0$, there is also a simple pole at $x=0$. Since
$\Tcal[R_0](0)=R_0(0)=(m-1)!$ and $L_{m-1}^{(1)}(0)=m$, its residue is
$1/m$, and therefore
\begin{equation}\label{eq:R0-partial-fraction}
\frac{\Tcal[R_0](x)}
{(m-1)!xL_{m-1}^{(1)}(-x)}
=
\frac1{mx}
+
\sum_{i=1}^{m-1}\frac{b_{i,0}}{x+x_i}.
\end{equation}

Applying $\Tcal$ to \eqref{eq:N-R} expresses the residues of
$T_{m,d,0}(x)/\Tcal[(y)_m](x)$ as positive linear combinations of the
numbers $b_{i,j}$. Their signs, sums, and ordering therefore control the
residues used to locate the zeros of $T_{m,d,\lambda}(x)$.

\begin{lemma}\label{lem:residues}
Let $m\ge2$. For $1\le i\le m-1$,
\(
b_{i,0}=\frac1m,
b_{i,1}=\frac1{m-1}.
\)
Every residue $b_{i,j}$ is positive. For $1\le j\le m-1$,
\begin{equation}\label{eq:b-sum}
\sum_{i=1}^{m-1}b_{i,j}=1.
\end{equation}
For $2\le j\le m-1$,
\begin{equation}\label{eq:b-monotone}
b_{1,j}<b_{2,j}<\cdots<b_{m-1,j}.
\end{equation}
\end{lemma}

\begin{proof}
For every polynomial $Q$, termwise comparison in \eqref{eq:Poisson-egf}
gives
\begin{align}
x\Tcal[Q(y)](x)
&=
\Tcal[yQ(y-1)](x),
\label{eq:T-shift}\\
x\frac{\dd}{\dd x}\Tcal[Q(y)](x)
&=
\Tcal\!\left[
y\bigl(Q(y)-Q(y-1)\bigr)
\right](x).
\label{eq:T-derivative}
\end{align}
For $1\le i\le m-1$, put $t_i=x_i^{-1}$.
For $1\le j\le m-1$, both $yR_j(y-1)$ and $R_{j-1}(y)$ are multiples of
$(y)_{m-1}/(y+j-1)$. To obtain a recurrence for $b_{i,j}$, we choose the
coefficients $m-j$ and $-j(j-1)$ so that this denominator cancels, leaving an
expression in $(y)_m$ and its first difference. Indeed, since
$(y)_m-(y-1)_m=m(y)_{m-1}$, the definition of $R_j$ gives
\begin{align*}
&(m-j)yR_j(y-1)-j(j-1)R_{j-1}(y)\\
&\quad=
\frac{(y)_{m-1}}{y+j-1}
\left((m-j)y(y-1)-j(j-1)(y+m-1)\right)\\
&\quad=
(y)_{m-1}\bigl((m-j)y-j(m-1)\bigr)\\
&\quad=
y\bigl((y)_m-(y-1)_m\bigr)-j(y)_m
\quad(0\le j\le m-1),
\end{align*}
where the term containing $R_{j-1}$ is omitted when $j=0$.
Applying $\Tcal$ and using
\eqref{eq:T-shift}--\eqref{eq:T-derivative} gives
\(
(m-j)x\Tcal[R_j](x)-j(j-1)\Tcal[R_{j-1}](x)
=x\frac{\dd}{\dd x}\Tcal[(y)_m](x)-j\Tcal[(y)_m](x),
\)
with the same convention when $j=0$. Evaluating at $x=-x_i$, dividing by
$-x_i(\Tcal[(y)_m])'(-x_i)$, and using the definition of $b_{i,j}$ gives
\(
mb_{i,0}=1
\)
and
\begin{equation}\label{eq:b-rec}
(m-j)b_{i,j}
=
1-j(j-1)t_i b_{i,j-1}
\qquad
(1\le j\le m-1).
\end{equation}
Thus $b_{i,0}=1/m$, and the case $j=1$ of \eqref{eq:b-rec} gives
$b_{i,1}=1/(m-1)$.

Let $e_k$ denote the elementary symmetric polynomial of degree $k$. For
$0\le k\le m-2$, put
\(
e_k^{(i)}
=e_k(t_1,\ldots,\widehat{t_i},\ldots,t_{m-1}).
\)
The normalized Laguerre polynomial factors as
\(
\frac{L_{m-1}^{(1)}(x)}m
=
\prod_{h=1}^{m-1}(1-t_hx).
\)
Comparing coefficients gives
\(
e_k(t_1,\ldots,t_{m-1})
=
\frac{(m-1)!}
{k!(k+1)!(m-1-k)!}
\quad
(0\le k\le m-1).
\)

We prove by induction that, for $1\le j\le m-1$,
\(
b_{i,j}
=
\frac{j!(j-1)!(m-1-j)!}{(m-1)!}\,e_{j-1}^{(i)}.
\)
The case $j=1$ is the formula for $b_{i,1}$ above. If $j\ge2$, put
\(
A_j=\frac{j!(j-1)!(m-j)!}{(m-1)!}.
\)
The induction hypothesis gives
\(
j(j-1)b_{i,j-1}
=A_je_{j-2}^{(i)}.
\)
The coefficient formula with $k=j-1$ gives
\(
A_je_{j-1}(t_1,\ldots,t_{m-1})=1,
\)
and
\(
e_{j-1}(t_1,\ldots,t_{m-1})
=e_{j-1}^{(i)}+t_i e_{j-2}^{(i)}.
\)
Substitution in \eqref{eq:b-rec} gives
\(
(m-j)b_{i,j}=A_je_{j-1}^{(i)},
\)
which completes the induction.

For $0\le k\le m-2$, we have $e_k^{(i)}>0$: the value for $k=0$ is $1$,
and for $k\ge1$ it is a nonempty sum of products of positive $t_h$.
The factorials in the formula for $b_{i,j}$ are positive, so $b_{i,j}>0$.
For $1\le j\le m-1$,
\eqref{eq:R-partial-fraction} and preservation of degree and leading
coefficient under $\Tcal$ give
\[
\frac{\Tcal[R_j](x)}{\Tcal[(y)_m](x)}
=\frac1x+O(x^{-2})
\quad(x\to\infty).
\]
On the other hand,
\[
\sum_{i=1}^{m-1}\frac{b_{i,j}}{x+x_i}
=\frac{\sum_{i=1}^{m-1}b_{i,j}}x+O(x^{-2}).
\]
Comparing the coefficients of $x^{-1}$ proves \eqref{eq:b-sum}.

Finally, since $t_1>\cdots>t_{m-1}>0$, for
$1\le i\le m-2$ and $2\le j\le m-1$,
\(
e_{j-1}^{(i+1)}-e_{j-1}^{(i)}
=
(t_i-t_{i+1})
e_{j-2}(t_1,\ldots,\widehat{t_i},\widehat{t_{i+1}},\ldots,t_{m-1})
>0.
\)
Together with the formula for $b_{i,j}$, this proves
\eqref{eq:b-monotone}.
\end{proof}

\subsection{The partial fraction expansion of
\texorpdfstring{$T_{m,d,\lambda}(x)$}{T(m,d,lambda)(x)}}

We first determine and estimate the residues contributed by $T_{m,d,0}(x)$.
Apply $\Tcal$ to \eqref{eq:N-R}. By \eqref{eq:R-partial-fraction} and
\eqref{eq:R0-partial-fraction}, there are unique real numbers
$h_1,\ldots,h_{m-1}$ such that
\begin{equation}\label{eq:uniform-rational}
\frac{T_{m,d,0}(x)}
{(m-1)!xL_{m-1}^{(1)}(-x)}
=
\frac1x
+
\sum_{i=1}^{m-1}
\frac{h_i}{x+x_i},
\end{equation}
where
\(
h_i
=
\sum_{j=0}^{m-1}
(m-j)\binom{d+j-1}{j}b_{i,j}.
\)

To count the real zeros of $T_{m,d,0}(x)$, we compare the residues
$h_1,\ldots,h_{m-1}$.

\begin{lemma}
\label{lem:h-order}
For $m\ge2$ and $d\ge1$, we have
\begin{equation}\label{eq:h-order}
d+1\le h_1<h_2<\cdots<h_{m-1}.
\end{equation}
\end{lemma}

\begin{proof}
The terms $j=0$ and $j=1$ contribute
$mb_{i,0}=m\cdot(1/m)=1$ and
$(m-1)db_{i,1}=(m-1)d\cdot(1/(m-1))=d$, respectively. This proves the
assertion when $m=2$. Suppose that $m\ge3$. Every remaining term is positive,
so $h_i>d+1$. The term with $j=2$ is strictly increasing in $i$ by
\Cref{lem:residues}, while all other terms are nondecreasing. This proves
\eqref{eq:h-order}.
\end{proof}

The correction term in \eqref{eq:T-lambda} changes both the constant term of
the partial fraction expansion and its residues, as follows.

\begin{proposition}
Let $m\ge2$, $d\ge1$, and $\lambda\in\R$.  Then
\begin{equation}\label{eq:T-rational}
\frac{T_{m,d,\lambda}(x)}
{(m-1)!xL_{m-1}^{(1)}(-x)}
=
-\frac{2\lambda}{d}
+
\frac1x
+
\sum_{i=1}^{m-1}
\frac{h_i-2\lambda/d}{x+x_i}.
\end{equation}
\end{proposition}

\begin{proof}
The Laguerre identities~\cite[Eqs.~18.9.13 and~18.9.23]{DLMF}
\[
L_{m-1}^{(2)}(-x)
=L_{m-1}^{(1)}(-x)+L_{m-2}^{(2)}(-x),
\qquad
\frac{\dd}{\dd x}L_{m-1}^{(1)}(-x)
=L_{m-2}^{(2)}(-x),
\]
give
\(
\frac{L_{m-1}^{(2)}(-x)}
{L_{m-1}^{(1)}(-x)}
=
1+
\frac{\dd}{\dd x}\log L_{m-1}^{(1)}(-x).
\)
Since
\(
L_{m-1}^{(1)}(-x)
=\frac1{(m-1)!}\prod_{i=1}^{m-1}(x+x_i),
\)
its logarithmic derivative yields
\[
\frac{L_{m-1}^{(2)}(-x)}
{L_{m-1}^{(1)}(-x)}
=
1+
\sum_{i=1}^{m-1}\frac1{x+x_i}.
\]
Combining this identity with \eqref{eq:T-lambda} and
\eqref{eq:uniform-rational} gives \eqref{eq:T-rational}. In particular,
$2\lambda/d$ is subtracted from the residue at every nonzero pole.
\end{proof}

\subsection{Real-rootedness of the Poisson transform of
\texorpdfstring{$N_{m,d,\lambda}(y)$}{N(m,d,lambda)(y)}}

The following lemma gives sufficient conditions for the numerator in a partial
fraction expansion of the form \eqref{eq:T-rational} to have distinct real
zeros.

\begin{lemma}
\label{lem:root-count}
Let $r\ge2$, $a>0$, and $0<x_1<\cdots<x_{r-1}$. Define the
polynomial $F$ and the rational function $\Phi$ by
\[
\Phi(x):=\frac{F(x)}
{x\prod_{i=1}^{r-1}(x+x_i)}
=
-a
+
\frac1x
+
\sum_{i=1}^{r-1}
\frac{c_i}{x+x_i}.
\]
Assume that
\(
c_1<c_2<\cdots<c_{r-1},
c_{r-1}>0,
c_i\ne0\quad(1\le i\le r-1).
\)
If $c_1<0$, assume in addition that there exists $x_*\in(-x_1,0)$ such
that $\Phi(x_*)>0$. Then $F$ has $r$ distinct real zeros.
\end{lemma}

\begin{proof}
The leading term of $F$ is $-ax^r$, so $\deg F=r$. Since the residues
$1,c_1,\ldots,c_{r-1}$ of $\Phi$ are nonzero, $F$ does not vanish at any
pole of $\Phi$, and no cancellation with the denominator occurs. The negative poles occur in the
order $-x_{r-1}<\cdots<-x_1<0$.

Suppose first that every $c_i$ is positive. Each of the $r-2$ intervals
between consecutive negative poles contains a zero of $F$, as does
$(-x_1,0)$.
Moreover, $\Phi(x)\to+\infty$ as $x\to0^+$, whereas
$\Phi(x)\to-a<0$ as $x\to+\infty$.
Thus $(0,\infty)$ contains a zero of $F$. These $r$ zeros are distinct.

Suppose now that some $c_i$ is negative. This case cannot occur when $r=2$,
because then $c_{r-1}=c_1>0$; hence $r\ge3$. There is a unique index $s$ such
that $c_s<0<c_{s+1}$. For $1\le i\le r-2$ with $i\ne s$, the residues
$c_i$ and $c_{i+1}$ have the same sign. Hence $\Phi$ has opposite signs
near the two ends of $(-x_{i+1},-x_i)$, so this interval contains a zero
of $F$. There are $r-3$ such intervals.

Since $c_1<0$,
\[
\lim_{x\to(-x_1)^+}\Phi(x)
=\lim_{x\to0^-}\Phi(x)=-\infty.
\]
Since $\Phi(x_*)>0$, each of $(-x_1,x_*)$ and $(x_*,0)$ contains a zero
of $F$. These two zeros, the zeros in the other $r-3$ intervals, and a
zero in $(0,\infty)$ give $r$ distinct real zeros of $F$.
\end{proof}

To bound the shift $2\lambda/d$ when applying \Cref{lem:root-count} to
\eqref{eq:T-rational}, we use the following packing estimate for sparse paving
matroids.

\begin{proposition}[{\cite[Proposition~5.15]{FNV2023}}]
\label{prop:packing}
Let $M$ be a sparse paving matroid of rank $d$ and corank $m$, where
$m,d\ge1$, and let $\lambda$ be the number of its circuit-hyperplanes. Then
\[
0\le\lambda\le
\Lambda_{m,d}
:=
\frac1{\max\{m,d\}+1}\binom{m+d}{d}.
\]
\end{proposition}

The residue estimates below use only the inequality
$\lambda\le\Lambda_{m,d}$, not the existence of a sparse paving matroid with
exactly $\lambda$ circuit-hyperplanes. We therefore allow any real
$\lambda\in[0,\Lambda_{m,d}]$.

By \Cref{lem:h-order}, the residues $h_i-2\lambda/d$ at the nonzero poles
remain strictly ordered; when $m=2$, there is only one such pole. To apply
\Cref{lem:root-count} to $T_{m,d,\lambda}$ when some of these residues are
negative, we must show that the residue at the leftmost pole is positive and
find a point of $(-x_1,0)$ at which the quotient in \eqref{eq:T-rational} is
positive.

\begin{lemma}
\label{lem:left-residue}
Suppose that $d\ge6$ and $m\ge4$. Then
$h_{m-1}>2\Lambda_{m,d}/d$.
\end{lemma}

\begin{proof}
By \eqref{eq:b-sum} and
$b_{1,j}\le\cdots\le b_{m-1,j}$,
we have
$b_{m-1,j}\ge1/(m-1)$ for $1\le j\le m-1$,
with strict inequality when $j\ge2$. Together with
$b_{m-1,0}=1/m$, we obtain
\( h_{m-1}
>
1+\frac1{m-1}
\sum_{j=1}^{m-1}
(m-j)\binom{d+j-1}{j}.
\)
Put $C=\binom{m+d}{m-1}$. Two applications of the binomial summation identity
yield
\(
\sum_{j=0}^{m-1}
(m-j)\binom{d+j-1}{j}
=
C.
\)
Since $C\ge\binom{m+1}{m-1}=m(m+1)/2>2$,
$dm-4(m-1)=(d-4)m+4>0$, and
\(
\frac Cm
=\frac1{d+1}\binom{m+d}{d}
\ge\Lambda_{m,d},
\)
we obtain
\(
h_{m-1}
>\frac{C-1}{m-1}
>\frac{C}{2(m-1)}
>\frac{2C}{dm}
\ge\frac{2\Lambda_{m,d}}d.
\)
\end{proof}

The following lemma gives a point of $(-x_1,0)$ at which the quotient in
\eqref{eq:T-rational} is positive.

\begin{lemma}
\label{lem:test-point}
Suppose that
\(
d\ge6,
m\ge4.
\)
Set $x_*=-2/(m+1)^2$.
Then
\(
-x_1<x_*<0,
\)
and, for every real $\lambda$ satisfying
$0\le\lambda\le\Lambda_{m,d}$,
\begin{equation}\label{eq:test-positive}
\frac{T_{m,d,\lambda}(x_*)}
{(m-1)!x_*L_{m-1}^{(1)}(-x_*)}
>0.
\end{equation}
\end{lemma}

\begin{proof}
Set $z=-x_*=2/(m+1)^2$. Let $N,\alpha$ be integers with
$0\le N\le m-1$ and $\alpha\ge1$. For $0<y\le z$ and each integer $i$
with $0\le i<N$,
the ratio of the absolute values of the terms indexed by $i+1$ and $i$
in the defining expansion of $L_N^{(\alpha)}(y)$ is
\(
\frac{(N-i)y}{(i+1)(\alpha+i+1)}
\le\frac{(m-1)z}{2}
=\frac{m-1}{(m+1)^2}<1.
\)
The terms alternate in sign, starting with a positive term, and strictly
decrease in absolute value. Pairing consecutive terms gives
$L_N^{(\alpha)}(y)>0$, with a positive final term left over when $N$ is
even. Since $L_N^{(\alpha)}(0)>0$ as well, we have
$L_N^{(\alpha)}(y)>0$ on $[0,z]$. In particular,
\(
z<x_1,\quad L_{m-1}^{(1)}(z)>0,\quad L_{m-1}^{(2)}(z)>0.
\)
Thus $-x_1<x_*<0$, and the denominator in
\eqref{eq:test-positive} is negative.

By \eqref{eq:T-lambda},
\(
T_{m,d,\lambda}(-z)
=T_{m,d,0}(-z)
+\frac{2\lambda}{d}(m-1)!zL_{m-1}^{(2)}(z).
\)
This expression is increasing in $\lambda$, so it suffices to prove
$T_{m,d,\Lambda_{m,d}}(-z)<0$.

\smallskip
\noindent
\emph{Finite cases: $4\le m\le9$ and $6\le d\le8$.}
Direct calculation gives $T_{m,d,\Lambda_{m,d}}(-z)<0$ for these eighteen pairs.

For all other pairs, either $d>m$ and $d\ge8$, or $m\ge d$ and $m\ge9$.
By \eqref{eq:N-uniform-def}, \eqref{eq:N-def}, and
\Cref{lem:Poisson-Laguerre},
\begin{equation}\label{eq:T-test-exact}
\begin{aligned}
T_{m,d,\Lambda_{m,d}}(-z)
&=m!-z\sum_{j=1}^{m-1}
(m-j)!\binom{m+d}{j}(j-1)!L_{j-1}^{(1)}(z)\\
&\qquad+\frac{2\Lambda_{m,d}}d(m-1)!zL_{m-1}^{(2)}(z).
\end{aligned}
\end{equation}
The alternating expansions also give
\[
\begin{gathered}
L_{m-r}^{(1)}(z)\ge(m-r+1)
\left(1-\frac{m-r}{(m+1)^2}\right)\qquad(r=2,3,4),\\
L_{m-1}^{(2)}(z)\le\binom{m+1}{2}.
\end{gathered}
\]
Keeping only $j=m-1,m-2,m-3$ in the sum in \eqref{eq:T-test-exact}
and applying these bounds, we obtain
\[
T_{m,d,\Lambda_{m,d}}(-z)
\le-z(m-1)!\binom{m+d}{m-1}\Delta_{m,d},
\]
where
\begin{equation}\label{eq:test-sufficient}
\begin{aligned}
\Delta_{m,d}
={}&\frac{m(m+1)(d^2+7d+18)+3d^2+23d+72}
{(m+1)^2(d+2)(d+3)}\\
&-\frac{\min\{m,d\}+1}{d}
-\frac{m(m+1)^2}{2\binom{m+d}{m-1}}.
\end{aligned}
\end{equation}
Thus $\Delta_{m,d}>0$ implies the required estimate.

\smallskip
\noindent
\emph{Case 1: $d>m$ and $d\ge8$.}
Let $F_m(x)$ be the right-hand side of \eqref{eq:test-sufficient}
with the final term omitted and $d$ replaced by a real variable $x>m$.
Differentiation gives
\[
F_m'(x)=\frac{m+1}{x^2}
-\frac{2(m^2+m+4)}{(m+1)^2(x+2)^2}
-\frac{6(m^2+m+5)(2x+5)}{(m+1)^2(x+2)^2(x+3)^2}.
\]
For $m\ge4$ and $x>m$, the differences
$(m+1)^2-(m^2+m+4)=m-3$ and $(m+1)^2-(m^2+m+5)=m-4$
are positive and nonnegative, respectively, while
$(x+3)^2-3(2x+5)=x^2-6>0$. Hence
\[
F_m'(x)>\frac{m+1}{x^2}-\frac4{(x+2)^2}
>\frac{m-3}{x^2}>0.
\]
The last term in \eqref{eq:test-sufficient} is negative, and the ratio of
its value at $d+1$ to its value at $d$ is
$(d+2)/(m+d+1)=1-(m-1)/(m+d+1)<1$.
It therefore increases with $d$, so $\Delta_{m,d}$ is strictly increasing in $d$.
The values of $\Delta_{m,8}$ for $4\le m\le7$ are positive:
\[
\Delta_{4,8}=\frac{3457}{11000},\quad
\Delta_{5,8}=\frac{7277}{25740},\quad
\Delta_{6,8}=\frac{8409}{40040},\quad
\Delta_{7,8}=\frac{667}{5720}.
\]
For $m\ge8$, it suffices to take $d=m+1$. Set
\[
R_m=\frac{2\binom{2m+1}{m-1}}{m^2(m+1)^2}.
\]
Since $R_8=2431/324>1$ and
\[
\frac{R_{m+1}}{R_m}-1
=\frac{3(m+1)(m-2)(m+2)+2m}{(m+3)(m+2)^2}>0
\qquad(m\ge8),
\]
we have $R_m>1$ throughout this range. Consequently,
\[
\Delta_{m,m+1}
=\frac1m\left(1-\frac1{R_m}\right)
+\frac{m^2(2m-3)+55(m-1)+43}{m(m+1)^2(m+3)(m+4)}>0.
\]

\smallskip
\noindent
\emph{Case 2: $m\ge d$ and $m\ge9$.}
Let $G_d(x)$ be the rational function obtained from
\eqref{eq:test-sufficient} by omitting the final term, setting
$\min\{m,d\}=d$, and replacing $m$ by $x$. For $x\ge9$,
\[
G_d'(x)
=\frac{(x-9)(d^2+7d+18)+4(d+3)^2}{(x+1)^3(d+2)(d+3)}>0.
\]
The last term in \eqref{eq:test-sufficient} increases strictly with $m$:
it is negative, and the ratio of its value at $m+1$ to its value at $m$ is
\[
\frac{(m+2)^2}{(m+1)(m+d+1)}
=1-\frac{(d-2)m+d-3}{(m+1)(m+d+1)}<1.
\]
Hence $\Delta_{m,d}$ is strictly increasing in $m$.
The values of $\Delta_{9,d}$ for $6\le d\le8$ are also positive:
\[
\Delta_{9,6}=\frac{433}{57200},\quad
\Delta_{9,7}=\frac{10991}{450450},\quad
\Delta_{9,8}=\frac{63953}{2431000}.
\]
For $d\ge9$, it suffices to take $m=d$. Set
\[
S_d=\frac{4\binom{2d}{d-1}}{d^3(d+1)^2}.
\]
Since $S_9=4862/2025>1$ and
\[
\frac{S_{d+1}}{S_d}-1
=\frac{(d-3)(3d^2+5d+3)+1}{(d+2)^3}>0
\qquad(d\ge9),
\]
we have $S_d>1$ for $d\ge9$, and therefore
\[
\Delta_{d,d}
=\frac2{d^2}\left(1-\frac1{S_d}\right)
+\frac{(d-4)(2d^3+d^2+19d+36)+132}
{d^2(d+1)^2(d+2)(d+3)}>0.
\]

Thus $T_{m,d,\Lambda_{m,d}}(-z)<0$ in every case.
Hence \eqref{eq:test-positive} holds.
\end{proof}

\begin{theorem}
\label{thm:T-real-rooted}
Let $m$ and $d$ be positive integers with $d\ge6$, and let
$\lambda\in[0,\Lambda_{m,d}]$. Then $T_{m,d,\lambda}(x)$ has only real
zeros.
\end{theorem}

\begin{proof}
The case $m=1$ follows from
$T_{1,d,\lambda}(x)=1-(2\lambda/d)x$. Assume $m\ge2$.

If $\lambda=0$, the residues in \eqref{eq:uniform-rational} are $1$ at the
origin and $h_i>0$ at $-x_i$. Thus no pole is canceled. Every interval between
consecutive nonzero poles contains a zero, as does $(-x_1,0)$. These are
$m-1$ distinct real zeros. Since the Poisson transform preserves degree and
$N_{m,d}(y)$ has degree $m-1$, they are all the zeros of $T_{m,d,0}(x)$.

Now suppose that $\lambda>0$, and set
\(
a_{\max}=\frac{2\Lambda_{m,d}}d.
\)
For $a\in(0,a_{\max}]$, define
\(
T_a(x)
=T_{m,d,0}(x)-a(m-1)!xL_{m-1}^{(2)}(-x).
\)
Then
\[
\frac{T_a(x)}{(m-1)!xL_{m-1}^{(1)}(-x)}
=-a+\frac1x+\sum_{i=1}^{m-1}\frac{h_i-a}{x+x_i}.
\]
Put $c_i=h_i-a$. The quotient allows \Cref{lem:root-count} to count the
zeros of $T_a$ from its poles. If $c_1>0$, all the shifted residues are
positive. If $c_1<0$, their strict ordering and the positivity of
$c_{m-1}$ give a single change of sign; the point supplied by
\Cref{lem:test-point} compensates for the one interval in which the poles do
not force a zero. A vanishing shifted residue will be handled by continuity.

Assume first that every $c_i$ is nonzero. If $c_1>0$, then every $c_i$ is
positive, and \Cref{lem:root-count}, with $r=m$ and $F=T_a$, shows that
$T_a$ is real-rooted.

Suppose that $c_1<0$. Then $a>h_1\ge d+1$ by \eqref{eq:h-order}, so
$ad/2>d^2/2$. Since $ad/2\le\Lambda_{m,d}$, we have
$\Lambda_{m,d}>d^2/2$. For $m=1,2,3$, respectively,
\(
\Lambda_{m,d}
=1,
\frac{d+2}{2},
\frac{(d+2)(d+3)}6.
\)
The largest of these is less than $d^2/2$, since
\[
\frac{d^2}{2}-\frac{(d+2)(d+3)}6
=\frac{(d-4)(2d+3)+6}{6}>0\qquad(d\ge6).
\]
Thus $m\ge4$, so we may apply \Cref{lem:left-residue,lem:test-point}.

The residues satisfy $c_1<\cdots<c_{m-1}$ by \eqref{eq:h-order}, and
\Cref{lem:left-residue}, together with $a\le a_{\max}$, gives
$c_{m-1}>0$. Moreover,
\Cref{lem:test-point}, applied with the real parameter $ad/2$, gives a point
$x_*\in(-x_1,0)$ such that
\(
\frac{T_a(x_*)}{(m-1)!x_*L_{m-1}^{(1)}(-x_*)}>0.
\)
All the hypotheses of \Cref{lem:root-count}, with $F=T_a$, are therefore
satisfied, so $T_a$ is real-rooted.

Finally, suppose that $a=h_i$ for some $i$. Choose
\(
a_\nu\longrightarrow a,
a_\nu\in(0,a_{\max}]\setminus\{h_1,\ldots,h_{m-1}\}.
\)
If $a=a_{\max}$, choose $a_\nu<a$. The preceding argument shows that every
$T_{a_\nu}$ is real-rooted, and its coefficients converge to those of $T_a$.
Since
\[
[x^m]T_{a_\nu}(x)=-a_\nu\longrightarrow-a\ne0,
\]
the limit has degree $m$. The continuity theorem for zeros of
polynomials~\cite[Section~1.3]{RahmanSchmeisser2002} shows that $T_a$ is
real-rooted. Taking $a=2\lambda/d$ proves the theorem.
\end{proof}

\subsection{Real-rootedness of
\texorpdfstring{$\Gamma_M(x)$ and $Z_M(x)$}{Gamma(M)(x) and Z(M)(x)}}

We first treat the cases $d=4$ and $d=5$ directly.

\begin{lemma}\label{lem:rank45-discriminants}
Let $m\ge1$, $d\in\{4,5\}$, and $0\le\lambda\le\Lambda_{m,d}$.
Then $\Gamma_{m,d,\lambda}(x)$ has only simple real zeros.
\end{lemma}

\begin{proof}
Since $\Gamma_{m,d,\lambda}(0)=1>0$ and $\Gamma_{m,d,\lambda}(x)$ has
degree at most two, it suffices to find a point where
$\Gamma_{m,d,\lambda}(x)<0$. We take $x=-2/(m+1)$ when $d=4$ and
$x=-1/(m+1)$ when $d=5$.

By \eqref{eq:sparse-gamma-formula}, the values
$\Gamma_{m,4,\lambda}(-2/(m+1))$ and
$\Gamma_{m,5,\lambda}(-1/(m+1))$ are affine functions of $\lambda$,
with coefficients $2(3m-1)/(m+1)^2$ and $4(m-1)/(m+1)^2$,
respectively. Both coefficients are nonnegative, so these values are
nondecreasing in $\lambda$. Thus the bound
\(
\lambda\le\Lambda_{m,d}\le\binom{m+d}{d}/(m+1),
\)
together with \eqref{eq:uniform-coeff} and \eqref{eq:sparse-gamma-formula},
gives
\begin{align*}
\Gamma_{m,4,\lambda}\!\left(-\frac2{m+1}\right)
&\le-\frac{m^2(m+1)(m+5)+6m(m-1)+12}{12(m+1)^2}<0,\\
\Gamma_{m,5,\lambda}\!\left(-\frac1{m+1}\right)
&\le-\frac{(m+2)(m+3)(m+4)\bigl(m(m-1)+10\bigr)+120}
{120(m+1)^2}<0.\qedhere
\end{align*}
\end{proof}

Finally, we are in a position to prove \Cref{thm:main}.

\begin{proof}[Proof of \Cref{thm:main}]
Let $d=\operatorname{rk}M$ and $m=|E(M)|-d$. If $d=0$, then
$\Gamma_M(x)=Z_M(x)=1$. If $m=0$, then $M$ is the free matroid,
$\Gamma_M(x)=1$, and $Z_M(x)=(1+x)^d$.
Assume that $m,d\ge1$, and let $\lambda$ be the number of circuit-hyperplanes
of $M$. Then $\Gamma_M(x)=\Gamma_{m,d,\lambda}(x)$ by the formula cited after
\eqref{eq:sparse-gamma-formula}.

Suppose first that $d\ge6$.
By \Cref{prop:packing}, $0\le\lambda\le\Lambda_{m,d}$.
The Poisson identity~\eqref{eq:Poisson-egf} reads
\(
\sum_{\ell\ge0}
N_{m,d,\lambda}(\ell)\frac{x^\ell}{\ell!}
=
e^xT_{m,d,\lambda}(x).
\)
Put $q=\lfloor d/2\rfloor$. The polynomial $T_{m,d,\lambda}(x)$ is
real-rooted by \Cref{thm:T-real-rooted}, and
$T_{m,d,\lambda}(0)=N_{m,d,\lambda}(0)=m!>0$.
Apply \Cref{thm:Jensen} to the displayed identity with $b=1$,
$P=T_{m,d,\lambda}$, and $\gamma_\ell=N_{m,d,\lambda}(\ell)$. The resulting
Jensen polynomial $\Jcal_q[N_{m,d,\lambda}](x)$ has constant term
$N_{m,d,\lambda}(0)=m!$, so it is nonzero and therefore real-rooted.

The $\gamma$-positivity theorem of Ferroni, Nasr, and
Vecchi~\cite[Theorem~5.16]{FNV2023} gives
\[[x^i]\Gamma_{m,d,\lambda}(x)\ge0\] for $0\le i\le q$. Since
$f_{m,d,i}>0$, the coefficient identity~\eqref{eq:Gamma-coefficients} yields
$N_{m,d,\lambda}(i)\ge0$ throughout this range.
Consequently, for $x\ge0$,
\[\Jcal_q[N_{m,d,\lambda}](x)\ge N_{m,d,\lambda}(0)=m!>0,\]
so all its zeros are negative.

By the criterion in \Cref{thm:finite-multiplier}, the operator
$N_{m,d,\lambda}(\vartheta)$ preserves real-rootedness for polynomials of
degree at most $q$. By \Cref{lem:H-simple-roots}, the polynomial $H_{m,d}(x)$
is real-rooted of degree $q$. Hence
$N_{m,d,\lambda}(\vartheta)H_{m,d}(x)$ is real-rooted, and
\eqref{eq:Gamma-Euler} identifies this polynomial with
$\Gamma_{m,d,\lambda}(x)$. Its coefficients are nonnegative and its constant
term is $1$, so all its zeros are strictly negative. By
\cite[Proposition~5.3]{FNV2023}, $Z_M(x)$ also has only negative real zeros.

It remains to consider $d\le5$. If $d\le3$, then
$\Gamma_{m,d,\lambda}(x)$ has degree at most one. For $d\in\{4,5\}$,
\Cref{prop:packing} gives $0\le\lambda\le\Lambda_{m,d}$, and
\Cref{lem:rank45-discriminants} shows that $\Gamma_{m,d,\lambda}(x)$ is
real-rooted. In all these cases, $\gamma$-positivity gives nonnegative
coefficients, and the
constant term is $1$, so every zero is negative. The equivalence of Ferroni,
Nasr, and Vecchi~\cite[Proposition~5.3]{FNV2023} now gives the same conclusion
for $Z_M(x)$.
\end{proof}

\section{Strict interlacing for uniform matroids}
\label{sec:uniform-interlacing}

Since uniform matroids are sparse paving, \Cref{thm:main} already shows that
every $\Gamma_{m,d}(x)$ has only negative real zeros. We prove the three
strict interlacing relations stated in
\Cref{thm:uniform-gamma-all-interlacings}. The comparison in which the corank
is fixed will then imply the corresponding strict interlacing relation for
the $Z$-polynomials in \Cref{thm:uniform-Z-strict-interlacing}.

For uniform matroids, \eqref{eq:Gamma-Euler} becomes
\[
\Gamma_{m,d}(x)=N_{m,d}(\vartheta)H_{m,d}(x).
\]
We first locate the zeros of $N_{m,d}(y)$. The three interlacing
comparisons then follow by combining these zero locations with
\Cref{lem:H-recursions,lem:H-simple-roots,lem:common-H-interlacing}.

\subsection{Zeros of \texorpdfstring{$N_{m,d}(y)$}{N(m,d)(y)}}

The interlacing criteria require the zeros of the polynomials defining the
Euler operators to occur in alternating order. To compare the zeros of
$N_{m,d}(y)$ as $m$ and $d$ vary, we label each zero by the interval
$(-j,-j+1)$ that contains it.
We first compare $N_{m,d}(y)$ with $N_{m,d+1}(y)$; the recursion
\eqref{eq:N-recursion} then gives the comparison with $N_{m+1,d}(y)$.

\begin{lemma}
\label{lem:N-uniform-boxes}
Let $m,d\ge1$. If $m=1$, then $N_{1,d}(y)=1$. If $m\ge2$, there are
unique zeros $\rho_{m,d,j}\in(-j,-j+1)$, $1\le j\le m-1$, and
\begin{equation}\label{eq:N-uniform-factorization}
N_{m,d}(y)
=m!\prod_{j=1}^{m-1}\left(1-\frac{y}{\rho_{m,d,j}}\right).
\end{equation}
Moreover,
\begin{equation}\label{eq:N-rank-zero-order-rho}
-j<\rho_{m,d,j}<\rho_{m,d+1,j}<-j+1
\qquad(1\le j\le m-1).
\end{equation}
\end{lemma}

\begin{proof}
The case $m=1$ follows from \eqref{eq:N-uniform-def}. Assume $m\ge2$. For
$0\le j\le m-1$, the calculation in the proof of
\Cref{prop:N-R-expansion} gives
$N_{m,d}(-j)=(-1)^j(m-j)!(d)_j$.
Set
\[
\Pi_m(y)=\prod_{j=0}^{m-1}(y-j),
\quad
a_{d,j}=\frac{(m-j)(d)_j}{j!}>0\quad(0\le j\le m-1),
\]
and
\[
\Phi_{m,d}(y)=\sum_{j=0}^{m-1}\frac{a_{d,j}}{y-j}.
\]
At $y=j$, the removable value of
$(-1)^{m-1}\Pi_m(y)\Phi_{m,d}(y)$ is
\[
(-1)^j(m-j)!(d)_j=N_{m,d}(-j).
\]
Lagrange interpolation at $0,1,\ldots,m-1$ yields
\[
N_{m,d}(-y)
=(-1)^{m-1}\Pi_m(y)\Phi_{m,d}(y).
\]

On each interval $(q-1,q)$, where $1\le q\le m-1$,
\(
\Phi_{m,d}'(y)
=-\sum_{j=0}^{m-1}\frac{a_{d,j}}{(y-j)^2}<0,
\)
and $\Phi_{m,d}$ decreases from $+\infty$ to $-\infty$. It therefore has
a unique zero $\beta_{m,d,q}\in(q-1,q)$. These give all the zeros
$\rho_{m,d,q}=-\beta_{m,d,q}$ of $N_{m,d}(y)$, and
\eqref{eq:N-uniform-factorization} follows from $N_{m,d}(0)=m!$.

For the dependence on $d$, note that
\[
a_{d+1,j}
=a_{d,j}\frac{d+j}{d}
=a_{d,j}+\frac jd a_{d,j}.
\]
Set $y=\beta_{m,d,q}$. Since $\Phi_{m,d}(y)=0$, we have
\begin{align*}
\Phi_{m,d+1}(y)
&=\Phi_{m,d}(y)
+\frac1d\sum_{j=0}^{m-1}\frac{ja_{d,j}}{y-j}\\
&=\frac1d\left(y\Phi_{m,d}(y)-\sum_{j=0}^{m-1}a_{d,j}\right)\\
&=-\frac1d\sum_{j=0}^{m-1}a_{d,j}<0.
\end{align*}
The second equality uses $j/(y-j)=y/(y-j)-1$. Since $\Phi_{m,d+1}$ is
strictly decreasing on $(q-1,q)$, its unique zero there lies to the left of
$\beta_{m,d,q}$. Negating this inequality gives
\eqref{eq:N-rank-zero-order-rho}.
\end{proof}

The same argument, using \eqref{eq:N-recursion}, gives the analogous
comparison in the corank parameter.

\begin{lemma}\label{lem:N-corank-zero-order}
For $m\ge2$, $d\ge1$, and $1\le j\le m-1$, we have
\begin{equation}\label{eq:N-corank-zero-order-rho}
\rho_{m,d,j}<\rho_{m+1,d,j}.
\end{equation}
\end{lemma}

\begin{proof}
Let $\rho=\rho_{m,d,j}$.  By \eqref{eq:N-recursion}, we get
\(
N_{m+1,d}(\rho)
=
\binom{m+d}{m}(\rho)_m.
\)
Since $\rho\in(-j,-j+1)$, exactly $j$ factors in
\(
(\rho)_m=\rho(\rho+1)\cdots(\rho+m-1)
\)
are negative.  Hence
\(
\operatorname{sgn}N_{m+1,d}(\rho)=(-1)^j.
\)
On the other hand, direct substitution gives
\(
N_{m+1,d}(-j+1)
=
(-1)^{j-1}(m-j+2)!(d)_{j-1},
\)
which has the opposite sign.  By \Cref{lem:N-uniform-boxes},
$N_{m+1,d}$ has exactly one zero in $(-j,-j+1)$, so this zero lies
strictly between $\rho$ and $-j+1$.  This proves \eqref{eq:N-corank-zero-order-rho}.
\end{proof}

\subsection{Strict interlacing of
\texorpdfstring{$\Gamma_{m,d}(x)$ and $\Gamma_{m,d+1}(x)$}
{Gamma(m,d)(x) and Gamma(m,d+1)(x)}}

We are now ready to prove \Cref{thm:uniform-gamma-all-interlacings}(i),
where the corank is fixed and the rank increases by one. This comparison uses the
rank recursion.

\begin{proof}[Proof of \Cref{thm:uniform-gamma-all-interlacings}(i)]
If $d=1$, then
\(
\Gamma_{m,1}(x)=1,
\Gamma_{m,2}(x)=1+mx.
\)
These polynomials are linearly independent, and the first has no zeros, so
they strictly interlace by our convention. We may therefore assume that
$d\ge2$.

Set
\(
q=\left\lfloor\frac{d+1}{2}\right\rfloor,
H(x)=H_{m,d+1}(x),
\)
and
\[
A(y)
=
\left(1-\frac{2y}{d+1}\right)N_{m,d}(y),
\qquad
B(y)=N_{m,d+1}(y).
\]
By \eqref{eq:H-rank-recursion} and \eqref{eq:Gamma-Euler},
\(
\Gamma_{m,d}(x)=A(\vartheta)H(x),
\Gamma_{m,d+1}(x)=B(\vartheta)H(x).
\)

If $m=1$, then $B(y)=1$ and
$-A(y)=2y/(d+1)-1$ have positive leading coefficients and strictly interlace
by our convention. If $m\ge2$, the zeros of $A$ are
\(
\rho_{m,d,m-1},\ldots,\rho_{m,d,1},
\frac{d+1}{2},
\)
and those of $B$ are
\(
\rho_{m,d+1,m-1},\ldots,\rho_{m,d+1,1}.
\)
By \eqref{eq:N-rank-zero-order-rho} and the locations of these zeros,
\(
\rho_{m,d,j}
<
\rho_{m,d+1,j}
<
\rho_{m,d,j-1}
\quad
(2\le j\le m-1),
\)
and
\(
\rho_{m,d,1}
<
\rho_{m,d+1,1}
<
\frac{d+1}{2}.
\)
Thus $A$ and $B$ have simple, strictly interlacing real zeros for every
$m\ge1$.

For $e\in\{d,d+1\}$, the coefficient of $x^{\lfloor e/2\rfloor}$ in
$\Gamma_{m,e}(x)$ is positive. Indeed, $e\ge2$ and
$N_{m,e}(\lfloor e/2\rfloor)\ge m!>0$, so \eqref{eq:uniform-coeff} gives
\(
\deg\Gamma_{m,e}(x)=\left\lfloor\frac e2\right\rfloor
\quad(e=d,d+1).
\)
Thus $A(\vartheta)H(x)=\Gamma_{m,d}(x)$ and
$B(\vartheta)H(x)=\Gamma_{m,d+1}(x)$ are nonconstant, and the description of the
zeros of $B$ above shows that $B$ has no zero in $[0,q]$.
By \Cref{lem:common-H-interlacing}, $\Gamma_{m,d}(x)$ and
$\Gamma_{m,d+1}(x)$ therefore strictly interlace.

If $d$ is odd, their degrees differ by one and the
interlacing convention fixes the stated direction. We may therefore assume
that $d$ is even, when their degrees are equal. Pascal's identity gives
\(
N_{m,d+1}(1)-N_{m,d}(1)
=
\sum_{j=1}^{m-1}
(m-j)!\binom{m+d}{j-1}(1)_j
\ge0.
\)
Since $N_{m,d}(1)\ge m!$, we obtain
\[
B(1)-A(1)
=N_{m,d+1}(1)-N_{m,d}(1)+\frac2{d+1}N_{m,d}(1)
\ge\frac{2m!}{d+1}>0.
\]
Since $[x]H(x)=d(d+1)/(m+1)!>0$, it follows that
\(
\Gamma_{m,d}'(0)<\Gamma_{m,d+1}'(0).
\)
\Cref{lem:interlacing-direction} fixes the order as stated.
\end{proof}

We now use \Cref{thm:uniform-gamma-all-interlacings}(i) to prove
\Cref{thm:uniform-Z-strict-interlacing}.

\begin{proof}[Proof of \Cref{thm:uniform-Z-strict-interlacing}]

Let $m,d\ge1$. If $d=1$, then
\(
Z_{U_{m,1}}(x)=1+x,
Z_{U_{m,2}}(x)=x^2+(m+2)x+1.
\)
The quadratic is positive at $0$ and for all sufficiently negative $x$, while
its value at $-1$ is $-m<0$. It therefore has one zero in $(-\infty,-1)$ and
one in $(-1,0)$, and hence
$Z_{U_{m,1}}(x)\prec Z_{U_{m,2}}(x)$. We may therefore assume that $d\ge2$.

By \Cref{thm:uniform-gamma-all-interlacings}(i),
\(
\Gamma_{m,d}(x)\prec\Gamma_{m,d+1}(x).
\)
The coefficient formula \eqref{eq:uniform-coeff} shows that their degrees are
$\lfloor d/2\rfloor$ and $\lfloor(d+1)/2\rfloor$, respectively. Applying
\Cref{prop:transformation-preserves-interlacing} with
$G_d(x)=\Gamma_{m,d}(x)$ and
$G_{d+1}(x)=\Gamma_{m,d+1}(x)$ gives
\(
Z_{U_{m,d}}(x)\prec Z_{U_{m,d+1}}(x),
\)
as required.
\end{proof}

\subsection{Strict interlacing of
\texorpdfstring{$\Gamma_{m,d}(x)$ and $\Gamma_{m+1,d}(x)$}
{Gamma(m,d)(x) and Gamma(m+1,d)(x)}}

We next prove \Cref{thm:uniform-gamma-all-interlacings}(ii), in which the
rank is fixed and the corank increases by one. Its proof uses the corank
recursion instead of the rank recursion.

\begin{proof}[Proof of \Cref{thm:uniform-gamma-all-interlacings}(ii)]
Set
\[
q=\left\lfloor\frac d2\right\rfloor,
\qquad H(x)=H_{m+1,d}(x),
\]
together with
\(
A(y)=(y+m+1)N_{m,d}(y),
B(y)=N_{m+1,d}(y).
\)
By \eqref{eq:H-corank-recursion} and \eqref{eq:Gamma-Euler},
\(
\Gamma_{m,d}(x)=A(\vartheta)H(x),
\Gamma_{m+1,d}(x)=B(\vartheta)H(x).
\)

For $m\ge2$, \Cref{lem:N-corank-zero-order} and the fact that each zero
lies in its interval give
\(
-(m+1)
<
\rho_{m+1,d,m}
<
\rho_{m,d,m-1}
<
\rho_{m+1,d,m-1}
<\cdots <
\rho_{m,d,1}
<
\rho_{m+1,d,1}.
\)
When $m=1$, the zero of $A$ is $-2$, while the zero of $B=N_{2,d}$
lies in $(-1,0)$. Thus $A$ and $B$ are nonconstant and have simple strictly
interlacing negative zeros. In particular, $B$ has no zero in $[0,q]$. By
\eqref{eq:uniform-coeff}, both
$A(\vartheta)H(x)=\Gamma_{m,d}(x)$ and
$B(\vartheta)H(x)=\Gamma_{m+1,d}(x)$ have degree $q\ge1$ and are therefore
nonconstant. By
\Cref{lem:common-H-interlacing}, they strictly interlace.

Moreover, \eqref{eq:N-recursion} gives
\(
B(1)-A(1)=\binom{m+d}{m}(1)_m
=m!\binom{m+d}{m}>0.
\)
It follows from $[x]H(x)=d(d-1)/(m+2)!>0$ that
\(
\Gamma_{m,d}'(0)<\Gamma_{m+1,d}'(0).
\)
The two $\gamma$-polynomials have the same degree, so
\Cref{lem:interlacing-direction} gives
\(
\Gamma_{m,d}(x)\prec\Gamma_{m+1,d}(x).
\)
\end{proof}

The hypothesis $d\ge2$ is necessary because both $\gamma$-polynomials are
equal to $1$ in rank one and hence do not strictly interlace under our
convention.

The comparison in which the rank is fixed does not give strict interlacing
of the corresponding $Z$-polynomials on the entire negative axis. Indeed,
$Z_{U_{m,d}}(x)$ and $Z_{U_{m+1,d}}(x)$ are palindromic of the same degree
$d$. Inversion $x\mapsto1/x$ preserves each of their zero sets and reverses
the order of the zeros. Two disjoint zero sets of the same cardinality
therefore cannot alternate on the entire negative axis.

Nevertheless, within each of the intervals $(-\infty,-1)$ and $(-1,0)$, the
zeros of $Z_{U_{m,d}}(x)$ and $Z_{U_{m+1,d}}(x)$ alternate strictly. To see
this, consider the map
\(
x\longmapsto\frac{x}{(1+x)^2}.
\)
It is a bijection from either interval onto $(-\infty,0)$. It reverses order
on $(-\infty,-1)$ and preserves order on $(-1,0)$. By
\eqref{eq:intro-gamma-transform}, the
zeros of each $Z$-polynomial other than $-1$ are exactly the preimages of the
zeros of the corresponding $\gamma$-polynomial. Therefore
\Cref{thm:uniform-gamma-all-interlacings}(ii) gives the asserted
alternation on both intervals. The ordering on $(-\infty,-1)$ is the reverse
of the ordering on $(-1,0)$. When $d$ is odd, both $Z$-polynomials also vanish
at $-1$.

\subsection{Strict interlacing of
\texorpdfstring{$\Gamma_{m,d}(x)$ and $\Gamma_{m+1,d+1}(x)$}
{Gamma(m,d)(x) and Gamma(m+1,d+1)(x)}}

We finally prove \Cref{thm:uniform-gamma-all-interlacings}(iii), in which
both the rank and corank increase by one. The proof combines the two
recursions used separately in the preceding subsections.

\begin{proof}[Proof of \Cref{thm:uniform-gamma-all-interlacings}(iii)]
If $d=1$, then
\(
\Gamma_{m,1}(x)=1,
\Gamma_{m+1,2}(x)=1+(m+1)x.
\)
These polynomials are linearly independent, and the first has no zeros, so
$\Gamma_{m,1}(x)\prec\Gamma_{m+1,2}(x)$ by our interlacing convention. We may
therefore assume that $d\ge2$.

Set
\[
\begin{gathered}
q=\left\lfloor\frac{d+1}{2}\right\rfloor,
\qquad
H(x)=H_{m+1,d+1}(x),
\\[1mm]
A(y)
=
(y+m+1)\left(1-\frac{2y}{d+1}\right)N_{m,d}(y),
\qquad
B(y)=N_{m+1,d+1}(y).
\end{gathered}
\]
By \eqref{eq:H-rank-recursion}, \eqref{eq:H-corank-recursion}, and
\eqref{eq:Gamma-Euler},
\(
\Gamma_{m,d}(x)=A(\vartheta)H(x),
\Gamma_{m+1,d+1}(x)=B(\vartheta)H(x).
\)

For $m\ge2$, \eqref{eq:N-rank-zero-order-rho} and
\eqref{eq:N-corank-zero-order-rho} give
\(
\rho_{m,d,j}
<
\rho_{m,d+1,j}
<
\rho_{m+1,d+1,j}
\quad
(1\le j\le m-1).
\)
Together with the fact that each zero lies in $(-j,-j+1)$, this yields
\(
-(m+1)
<
\rho_{m+1,d+1,m}
<
\rho_{m,d,m-1}
<
\rho_{m+1,d+1,m-1}
<\cdots <
\rho_{m,d,1}
<
\rho_{m+1,d+1,1}
<
\frac{d+1}{2}.
\)
For $m=1$, the zeros occur in the order
\(
-2<\rho_{2,d+1,1}<\frac{d+1}{2}.
\)
Thus the nonconstant polynomials $A$ and $B$ have simple strictly
interlacing real zeros. The
polynomial $B$ has no zero in $[0,q]$. By \eqref{eq:uniform-coeff},
$A(\vartheta)H(x)=\Gamma_{m,d}(x)$ and
$B(\vartheta)H(x)=\Gamma_{m+1,d+1}(x)$ have positive degrees
$\lfloor d/2\rfloor$ and $\lfloor(d+1)/2\rfloor$, respectively. Hence
\Cref{lem:common-H-interlacing} gives strict interlacing of
$\Gamma_{m,d}(x)$ and $\Gamma_{m+1,d+1}(x)$.

If $d$ is odd, their degrees differ by one, so the
direction is
fixed by the interlacing convention.  Suppose that $d$ is even.  Then the
degrees are both $d/2$, and Pascal's identity gives
\(
N_{m,d+1}(1)-N_{m,d}(1)
=
\sum_{j=1}^{m-1}
(m-j)!\binom{m+d}{j-1}(1)_j
\ge0.
\)
Using \eqref{eq:N-recursion} and $N_{m,d}(1)\ge m!$, we obtain
\[
B(1)-A(1)
\ge m!\left(\frac{2(m+2)}{d+1}+\binom{m+d+1}{m}\right)>0.
\]
Together with $[x]H(x)=d(d+1)/(m+2)!>0$, this gives
\(
\Gamma_{m,d}'(0)<\Gamma_{m+1,d+1}'(0).
\)
The stated order follows from \Cref{lem:interlacing-direction}.
\end{proof}

\begin{remark}
The comparison along the other diagonal,
$\Gamma_{m+1,d}(x)\preccurlyeq\Gamma_{m,d+1}(x)$, need not hold. For example,
\(
\Gamma_{2,3}(x)=1+7x,
\Gamma_{1,4}(x)=1+6x+2x^2.
\)
The zero $-1/7$ of $\Gamma_{2,3}(x)$ lies to the right of both zeros of
$\Gamma_{1,4}(x)$, so
\(
\Gamma_{2,3}(x)\not\preccurlyeq\Gamma_{1,4}(x).
\)
\end{remark}

\section*{Acknowledgements}
All authors contributed equally and the names of the authors are listed in
alphabetical order according to their family names. Matthew H.~Y.~Xie was
supported by the National Natural Science Foundation of China (Grant
No.~12271403). Philip B.~Zhang was supported by the National Natural Science
Foundation of China (Grant No.~12171362) and Natural Science Foundation of
Tianjin Municipality (Grant No.~25JCYBJC00430).

\section*{Declaration of Generative AI}
During the preparation of this manuscript, the authors used generative AI
tools to assist in exploring possible approaches, checking technical details,
and improving the exposition. All mathematical arguments, computations,
proofs, and results were independently verified by the authors. The authors
take full responsibility for the content of this manuscript.

\end{document}